\documentclass[11pt,openright]{article}	
\usepackage[margin=1.2in]{geometry}
\usepackage{authblk}

\usepackage{amscd,amsfonts,amssymb,amsmath,amsthm,latexsym}                                                             
\usepackage[pctex32]{graphics}			
\usepackage{times}				        
\usepackage[]{amssymb}				    
\usepackage[]{amsfonts}				    
\usepackage[]{amsmath}			     	
\providecommand{\bysame}			    
{\makebox[3em]{\hrulefill}\thinspace}	
\usepackage[Sonny]{fncychap}   		

\newtheorem{definition}{Definiton}[section]
\newtheorem{remark}{Remark}[section]

\newtheorem{lemma}{Lemma}[section]
\newtheorem{theorem}{Theorem}[section]
\newtheorem{corolary}{Corollary}[section]

\usepackage[]{fancyhdr}			                            
\fancypagestyle{plain}{					        
   \fancyhead{}                         	    
   }		    

\usepackage{graphicx}
\graphicspath{%
    {converted_graphics/}
    {/}
}

\title{Random Sturm-Liouville Operators with Point Interactions}
\date{}

\author{Rafael del Rio\footnote{delriomagia@gmail.com} \, and Asaf L. Franco \footnote{asaflevif@gmail.com}\\
IIMAS - UNAM$^{*\dagger}$\\
Circuito escolar, Ciudad universitaria 04510 CDMX, México
}

\begin{document}

\maketitle

\begin{abstract}
     We study invariance for eigenvalues of selfadjoint Sturm-Liouville operators with local point interactions. Such linear transformations are formally defined by 
$$H_{\omega}:=-\frac{d^2}{dx^2}+V(x)+\sum_{n\in I}\omega(n)\delta(x-x_n)$$
or similar expressions with $\delta'$ instead of $\delta$. In a probabilistic setting, we show that a point is either an eigenvalue for all $\omega$ or only for a set of $\omega$'s of measure zero. Using classical oscillation theory it is possible to decide whether the second  situation happens. The operators do not need to be measurable or ergodic. This generalizes the well known fact that for ergodic operators a point is eigenvalue with probability zero.
\end{abstract}

Mathematics Subject Classification (MSC2010): 34L05, 47E05, 47N99.

\section{Introduction}
This work is about point spectra of selfadjoint  Sturm-Liouville operators with $\delta,\delta'$-interactions. These are  defined by expressions of the form 
$$H_{\omega}:=-\frac{d^2}{dx^2}+V(x)+\sum_{n\in I}\omega(n)\delta(x-x_n)$$
or  with $\delta'$ instead of $\delta$. There are several way to introduce this objects. They can be constructed by using form methods, see \cite{KM} or by adding boundary conditions as in \cite{EVZ}, for example. Here  we shall use the approach developed in \cite{BSW} which generalizes Sturm-Liouville classical theory to include local point interactions. This has the advantage that selfadjointness, the Weyl alternative  and related  results can be established along the lines of a well known theory.
For a detailed study of this field, including many solvable models in quantum mechanics as well as an extensive list of references see the monograph \cite{AGHK}.\\

 The relations between the operators and their spectra, have deep consequences in several areas of Functional analysis, Scattering theory, Localization problems,  Dynamic behavior of Quantum systems, Differential and Integral equations, Matrix theory and so on. We shall   focus on the point spectrum and consider  operators generated by $\delta$ or $\delta'$ interactions with one common eigenvalue. This can be regarded as an inverse spectral problem, where given a point $E\in\mathbb{R}$ one tries to characterize  the sequences $\omega$ for which  $E$ belongs to the point spectra of the  operators $H_\omega$. The way we proceed is by analyzing first the operator with just one point interaction, then extend the results obtained in this case to a countable number of interactions. Finally placing our operators in a random environment, we are able to give the characterization of operators sharing the same eigenvalue in a probabilistic setting. \\
 
 In the random situation we consider here, the $\omega$ associated to $H_\omega$ is a stochastic process and each $\omega(n)$  a random variable with continuous (maybe singular) probability distribution. Our operators $H_\omega$ do not have to be measurable (see Definition \ref{meafam} ) and $\omega$ does not have to be a stationary metrically transitive random field or ergodic, see Section 9.1 \cite{CFKS}. For metrically transitive random operators it is well known that the probability for a given $E \in\mathbb{R}$ to be an eigenvalue is zero (see Corollary 1 Section 4.3 \cite{WK}, Theorem 2.12  \cite{PF}).  If we do not have this condition, in principle any situation could be possible. We show that for  random operators with point interactions , the following alternative holds: a point is either an eigenvalue for all $\omega$ or only for a set of $\omega$'s of measure zero. To decide which of these situations happens we were able to use classical oscillation theory, exploiting the relation between the zeros of eigenfunctions and the placement of the points interactions.  \\

This work is organized as follows. In Section \ref{1delta} we consider the operator with only one interaction in the regular case and  study the behavior of its point spectrum. A key tool is the relation between the Green's function associated to different boundary conditions. Starting from classic solutions we construct more general ones for the problem at hand. In Section \ref{ndelta} the results obtained for one interaction are extended to the case of countably many  and  the operators generated by the corresponding formal differential expressions are introduced. In Section \ref{mdelta} we apply the results of Sections \ref{1delta} and \ref{ndelta} to Random Operators . Theorem \ref{isornot} gives  then a characterization of the $\omega$'s such that $H_\omega$ share an eigenvalue.  Subsection \ref{sub1} considers zeros of eigenfunctions belonging to the operator without point interactions. It is proven in particular, that  nonoscillatory behavior implies the family $H_\omega$'s has a common eigenvalue for a set of $\omega$'s of measure zero.  Analogous results hold if the interactions are placed close enough. In subsection \ref{sub2}  measurable operators are introduced. Finally, in Section \ref{delta2} we study operators with $\delta'-$interactions  and show that similar results to the ones with $\delta$ holds.\\

We denote as usual by $\mathbb{R}$ the real numbers, by $L^1(J):=\{f:J\rightarrow\mathbb{R}: \int_J |f|<\infty\}$ the integrable functions , $L_{loc}^1(J):=\{ f\in L^1(\tilde J): \tilde J\subseteq J, \, \tilde J \mbox{ a closed interval}     \} $ the local integrable functions and the eigenvalues of an operator $L$ by $\sigma_p(L)$.


	\section{Sturm-Liouville Operators With One $\delta$-Point Interaction }\label{1delta}
	First we analyze the case of only one interaction in the regular case. We develop the basic steps that would be used in upcoming sections.
	Let $J\subset \mathbb R$ a closed finite interval. Let $V\in L^1(J)$ real valued function, $p\in J$ an interior point and $\alpha\in\mathbb{R}$. We consider the  formal differential expressions 
	$$\tau:=-\frac{d^2}{dx^2}+V$$
	$$\tau_{\alpha,p}:=-\frac{d^2}{dx^2}+V+\alpha\delta(x-p)$$
	The maximal operator $T_{\alpha,p}$  corresponding to $\tau_{\alpha,p}$ is defined by 
	$$T_{\alpha,p} f=\tau f$$
	$$D(T_{\alpha,p})=\{f\in L^2(J):\,f,\,f'\mbox{ abs. cont in }J\backslash \{p\},-f''+Vf\in L^2(J),$$
	$$f(p+)=f(p-),\,f'(p+)-f'(p-)=\alpha f(p)\}$$
	
	We extend the concept of solution and Wronskian in the following way
	
	\begin{definition}\label{sol}
Given $g\in L^1(J)$ and $z\in\mathbb{C}$, we call $f$ a solution of $(\tau_{\alpha,p}-z)f=g$ if $f$ and $f'$ are absolutely continuous in $J\backslash\{p\}$ with $-f''+Vf-zf=g$ and $f(p+)=f(p-)$, $f'(p+)-f'(p-)=\alpha f(p)$.
\end{definition}

\begin{definition}\label{wron}
Let $u_1$ and $u_2$ solutions of  $(\tau_{\alpha,p}-z)u=0$, $z\in\mathbb{C}$. The Wronskian $W(u_1,u_2)$ is defined by
		$$W_x(u_1,u_2)=W(u_1,u_2)(x)=det\left(\begin{array}{cc}u_1(x+)&u_2(x+)\\
		                                                                                            u'_1(x+)&u'_2(x+)\end{array}\right)=u_1(x+)u'_2(x+)-u'_1(x+)u_2(x+)$$
	\end{definition}
	
	\begin{lemma}\label{wcont}
		$W(u_1,u_2)$ is continuous at $p$.
	\end{lemma}	
\[\begin{array}{lcl}
W_{p-}(u_1,u_2)-W_{p+}(u_1,u_2)&=&u_1(p)u_2'(p-)-u_1'(p-)u_2(p)-u_1(p)u_2'(p+)+u_1'(p+)u_2(p)\\
                                                   &=&u_1(p)[u_2'(p-)-u_2'(p+)]+u_2(p)[u_1'(p+)-u_1'(p-)]\\
                                                   &=&u_1(p)[-\alpha u_2(p)]+u_2(p)[\alpha u_1(p)]\\
                                                   &=&0
\end{array}\]

\qed

 Let us fix $J=[a,b]$

	\begin{lemma}\label{wdop}
	Let $u$ and $v$ solutions of $$\tau_{\alpha,p} u=\lambda_0 u \qquad\qquad \mbox{y}\qquad\qquad \tau_{\alpha,p} v=\lambda v$$ respectively. Let $c,d\in[a,b]\backslash\{p\} $. Then
	$$W_d(u,v)-W_c(u,v)=(\lambda_0-\lambda)\int_c^du(t)v(t)dt.$$ 	
	\end{lemma}
	\textit{Proof. }
	\[
	\begin{array}{rcl}
	\frac{d}{dx}W_x(u,v)&=&\frac{d}{dx}[u(x)v'(x)-u'(x)v(x)]\\\\
	&=&u'(x)v'(x)+u(x)v''(x)-u''(x)v(x)-u'(x)v'(x)\\\\
	&=&u(x)[V(x)v(x)-\lambda v(x)]-[V(x)u(x)-\lambda_0 u(x)]v(x)\\\\
	&=&(\lambda_0-\lambda)u(x)v(x)
	\end{array}
	\]
	Then
	$$\frac{d}{dx}W_x(u,v)=(\lambda_0-\lambda)u(x)v(x)\qquad\qquad \forall x\in [a,b]\backslash \{p\}$$
	Let $c,d\in[a,b]\backslash\{p\} $. If $a\leq c<d<p$ or $p<c<d\leq b$, then from the fundamental theorem of calculus
	$$W_d(u,v)-W_c(u,v)=(\lambda_0-\lambda)\int_c^d u(t)v(t)dt$$
	If $a\leq c<p<d\leq b$, by continuity of the Wronskian at $p$, Lemma \ref{wcont}, we have
	$$(\lambda_0-\lambda)\int_c^d u(t)v(t)dt=(\lambda_0-\lambda)\left[\int_c^p u(t)v(t)dt+\int_p^d u(t)v(t)dt\right]$$
	$$=W_{p-}(u,v)-W_c(u,v)+W_d(u,v)-W_{p+}(u,v)=W_d(u,v)-W_c(u,v)$$
	\qed\\

\begin{definition}\label{solution}	 
	A solution of $\tau_{\alpha,p} u=\lambda u$ which satisfies at a  point $l\in [a,b]\backslash \{p\} $ the boundary condition 
	$$u(l)cos\theta+u'(l)sen\theta=0,\qquad \theta\in[0,\pi)$$
	will be denoted by $u_{l,\alpha}(\lambda)$.
\end {definition}

Such solution can be constructed as follows. Assume $l\in[a,p)$, choose $w_1(x,\lambda)$ solution of $(\tau-\lambda)u=0$, such that
$$w_1(l,\lambda)=\sin\theta$$ $$w'_1(l,\lambda)=-\cos\theta$$
Once we have $w_1(x,\lambda)$, choose $w_2(x,\lambda)$ solution of $(\tau-\lambda)u=0$, such that
$$w_2(p,\lambda)=w_1(p,\lambda)$$
$$w'_2(p,\lambda)=w'_1(p,\lambda)+\alpha w_1(p,\lambda)$$
Then we can set 

$$u_{l,\alpha} (x,\lambda)=\left\{\begin{array}{c}
w_1(x,\lambda)\qquad\mbox{ if }x\leq p\\
w_2(x,\lambda)\qquad\mbox{ if }x> p
\end{array}\right.
$$
In case $l\in (p,b]$, the construction is analogous. If $\alpha=0$ then we have a solution for the classic case.\\

The functions $u_{l,\alpha}(x,\lambda)$ and $u'_{l,\alpha}(x,\lambda)$ are entire with respect to $\lambda$ for each fixed $x\in[a,b]$. See \cite{GT}, Theorem 9.1 and \cite{AZ}, Theorem 2.5.3.\\

The following is a generalization of Theorem 8.4.2 in \cite{ATK},\\\\
\begin{theorem}\label{atkin}
	Let $u_{a,\alpha}(\lambda)$ as in the above definition.
If $u'_{a, \alpha}(\lambda,x)\not=0$, then $\forall x\in[a,b]\backslash \{p\}$
$$\frac{\partial}{\partial\lambda}\left\{\frac{u_{a,\alpha}(\lambda,x)}{u'_{a,\alpha}(\lambda,x)}\right\}=\frac{1}{u'_{a,\alpha}(\lambda,x)^2}\int_a^x u_{a,\alpha}(\lambda,t)^2dt$$
and if $u_{a,\alpha}(\lambda,x)\not=0$
$$\frac{\partial}{\partial\lambda}\left\{\frac{u'_{a,\alpha}(\lambda,x)}{u_{a,\alpha}(\lambda,x)}\right\}=-\frac{1}{u_{a,\alpha}(\lambda,x)^2}\int_a^x u_{a,\alpha}(\lambda,t)^2dt$$
		\end{theorem}

	\textit{Proof.}	
If in Lemma \ref{wdop} we choose $u=u_{a,\alpha}(\lambda)$ and $v=u_{a,\alpha}(\tilde\lambda)$ then, for $x\in[a,b]\backslash\{p\}$ 
	 
 $$W_x(u_{a,\alpha}(\lambda),u_{a,\alpha}(\tilde\lambda))=(\lambda-\tilde\lambda)\int_a^x u_{a,\alpha}(\lambda,t)u_{a,\alpha}(\tilde\lambda,t)dt$$
since $W_a(u_{a,\alpha}(\lambda),u_{a,\alpha}(\tilde\lambda))=0$ 
	
Thus $\forall x\in[a,b]\backslash \{p\}$
	
	\[
	\begin{array}{rcl}
	\int_a^xu_{a,\alpha}(\lambda,t)u_{a,\alpha}(\tilde\lambda,t)dt&=&\frac{u_{a,\alpha}(\lambda)u'_{a,\alpha}(\tilde\lambda)-u'_{a,\alpha}(\lambda)u_{a,\alpha}(\tilde\lambda)}{\lambda-\tilde\lambda}\\\\
	
	&=&\frac{u_{a,\alpha}(\lambda)u'_{a,\alpha}(\tilde\lambda)-u_{a,\alpha}(\lambda)u'_{a,\alpha}(\lambda)+u_{a,\alpha}(\lambda)u'_{a,\alpha}(\lambda)-u'_{a,\alpha}(\lambda)u_{a,\alpha}(\tilde\lambda)}{\lambda-\tilde\lambda}\\\\
	
	&=&u'_{a,\alpha}(\lambda)\frac{u_{a,\alpha}(\lambda)-u_{a,\alpha}(\tilde\lambda)}{\lambda-\tilde\lambda}-u_{a,\alpha}(\lambda)\frac{u'_{a,\alpha}(\lambda)-u'_{a,\alpha}(\tilde\lambda)}{\lambda-\tilde\lambda}
	
	\end{array}
	\] 
	
	Letting $\tilde\lambda\rightarrow\lambda$, we get
		$$\int_a^x u_{a,\alpha}(\lambda,t)^2dt=u'_{a,\alpha}(\lambda,x)\frac{\partial}{\partial\lambda}u_{a,\alpha}(\lambda,x)-u_{a,\alpha}(\lambda,x)\frac{\partial}{\partial\lambda}u'_{a,\alpha}(\lambda,x)$$
	Dividing by $u'_{a,\alpha}(\lambda,x)^2$ we get the first equality and dividing by $u_a(\lambda,x)^2$ we obtain the second.
	\\\hfill\qed

\begin{definition}\label{Ge}
Let us define for $z\in\mathbb{C}$ 
$$G_\alpha(z,x,x)=\frac{u_{a,\alpha}(z,x)u_{b,\alpha}(z,x)}{W_x(u_{a,\alpha}(z),u_{b,\alpha}(z))}$$
\end{definition}
This happens to be the Green function of a selfadjoint operator, but we shall not use that.\\

For the next Theorem we got important input from G. Teschl.
\begin{theorem}\label{compa}
	For any $\alpha\in\mathbb{R}$ we have
	$$G_\alpha(z;p,p)=G_0(z;p,p)\frac{1}{1-\alpha G_0(z;p,p)}$$
\end{theorem}
\textit{Proof.} 
If $x\leq p$, $u_{a,0}(x)=u_{a,\alpha}(x)$ and if $x\geq p$, $u_{b,0}(x)=u_{b,\alpha}(x)$. \\\\
Now, from the condition at $p$
$$u'_{a,\alpha}(p+)=u'_{a,\alpha}(p-)+\alpha  u_{a,\alpha}(p)=u'_{a,0}(p-)+\alpha  u_{a,0}(p)=u'_{a,0}(p+)+\alpha u_{a,0}(p).$$

Using this in $G_\alpha$ we get
$$G_\alpha(z,p,p)=\frac{u_{a,\alpha}(p) u_{b,\alpha}(p)}{W( u_{a,\alpha}, u_{b,\alpha})}=\frac{ u_{a,\alpha}(p) u_{b,\alpha}(p)}{u_{a,\alpha}(p) u'_{b,\alpha}(p+)-u'_{a,\alpha}(p+) u_{b,\alpha}(p)}=$$
$$=\frac{u_{a,0}(p)u_{b,0}(p)}{u_{a,0}(p)u'_{b,0}(p+)-u'_{a,0}(p+)u_{b,0}(p)-\alpha u_{a,0}(p)u_{b,0}(p)}
$$
Then
$$G_\alpha(z,p,p)=\frac{u_{a,0}(p)u_{b,0}(p)}{W(u_{a,0},u_{b,0})\left(1-\alpha\frac{u_{a,0}(p)u_{b,0}(p)}{W(u_{a,0},u_{b,0})}\right)}=G_0(z,p,p)\frac{1}{1-\alpha G_0(z,p,p)}$$
\hfill\qed

\begin{corolary}\label{corocompa}
	If $G_\alpha:=G_\alpha(z;p,p)$, then $\forall \alpha,\beta\in\mathbb{R}$, $\alpha\not=\beta$, 
	$$G_\beta=\frac{G_\alpha}{1+(\alpha-\beta)G_\alpha}.$$
\end{corolary}
\noindent\textit{Proof.} From theorem \ref{compa}, if $G_0=0$, $G_\alpha=0$, $\forall\alpha\in\mathbb{R}$. If $G_0\not=0$,
$$G_\alpha=\frac{1}{\frac{1}{G}-\alpha}.$$
Then
$$\frac{1}{G_0}=\frac{1}{G_\alpha}+\alpha=\frac{1}{G_\beta}+\beta$$
Hence
$$G_\beta=\frac{G_\alpha}{1+(\alpha-\beta)G_\alpha}.$$
\hfill\qed

	Suppose now $\tau_{\alpha,p}$ regular at $a$ and $b$, i.e $a$ and $b$ finite, $V\in L^1([a,b])$. Let us consider the selfadjoint restriction $H_{\alpha,p}$ of $T_{\alpha,p}$ in $L_2(a,b)$, see Theorem 5.2 in \cite{BSW}, defined by
	\begin{equation}\label{halfa}
	H_{\alpha,p} f=\tau f\end{equation}
	\[
	\begin{array}{ccc}
	D(H_{\alpha,p})&=&\left\{f\in D(T_{\alpha,p}):
	\begin{array}{c}
	f(a)cos\theta+f'(a)sen\theta=0 \\ 
	{f(b)cos\gamma+f'(b)sen\gamma=0}
	\end{array}
	\right\}\qquad\qquad  \theta,\,\gamma\in [0,\pi).
	\end{array}\]

\begin{theorem}\label{polo}
	Let $E$ eigenvalue of $H_{\alpha,p}$, then $G_\alpha(E,p,p)=0$ or $G_\alpha(z,p,p)$ has a pole in $E$.
\end{theorem}
\textit{Proof.} Let $u_{a,\alpha}(E,x)$ and $u_{b,\alpha}(E,x)$ solutions of $(H_{\alpha,p}-E)u=0$ which satisfy the boundary conditions at $a$ and $b$ respectively. Then $u_{a,\alpha}$ and $u_{b,\alpha}$ are linearly dependent and $W(u_{a,\alpha}(E),u_{b,\alpha}(E))=0$. See \cite{BSW}, Lemma 4.2.

\begin{itemize}
\item If $u_{a,\alpha}(E,p)\not=0$, then
\begin{equation}\label{case1}
\lim_{z\rightarrow E}|G_\alpha(z,p,p)|=\lim_{z\rightarrow E}\left|\frac{u_{a,\alpha}(z,p)u_{b,\alpha}(z,p)}{W(u_{a,\alpha}(z),u_{b,\alpha}(z))}\right|=\infty
\end{equation}

\item Now we consider the case $u_{a,\alpha}(E,p)=0$.  Let $u_{b,\alpha}$ be a solution such that $u_{b,\alpha}(E,b)=-\sin\theta$ and $u_{b,\alpha}'(E,b)=cos\theta$, $\theta\in[0,\pi)$. Since, for each $z$ fixed, $W_x(u_{a,\alpha}(z),u_{b,\alpha}(z))$ is constant for all $x$ in $[a,b]$,then
$$W_x(u_{a,\alpha}(E),u_{b,\alpha}(E))=W_b(u_{a,\alpha}(E),u_{b,\alpha}(E))=u_{a,\alpha}(E,b)cos\theta+ u'_{a,\alpha}(E,b)\sin\theta$$

 The functions $u_{a,\alpha}'(E,b)$ and $u_{a,\alpha}(E,b)$ cannot vanish simultaneously. Assume for example that  $u_{a,\alpha}'(E,b)\not=0$. 
 $$\frac{\partial}{\partial\lambda}\left[\frac{W_p(u_{a,\alpha}(\lambda),u_{b,\alpha}(\lambda))}{u'_{a,\alpha}(\lambda,b)}\right]_{\lambda=E}=$$
 $$=\frac{\partial}{\partial\lambda}\left[\frac{u_{a,\alpha}(\lambda,b)}{u'_{a,\alpha}(\lambda,b)}cos\theta+\sin\theta\right]_{\lambda=E}=\frac{cos\theta}{u'_{a,\alpha}(E,b)^2}\int_a^bu_{a,\alpha}(E,t)^2dt\not=0$$
where for the last equality Theorem \ref{atkin} was used.  Since we are assuming that $E$ is an eigenvalue, the functions $u_{a,\alpha}(E,b)$ and $u_{b,\alpha}(E,b)$ are linearly dependent. Hence $u_{a,\alpha}'(E,b)=Cu_{b,\alpha}'(E,b)=C\cos\theta\not=0$.
Therefore
$$\frac{W_p(u_{a,\alpha}(\lambda),u_{b,\alpha}(\lambda))}{u'_{a,\alpha}(\lambda,b)}$$
 has  a zero  at $E$ of order one and since $u_{a,\alpha}'(E,b)\not=0$ so does  $W_p(u_{a,\alpha}(\lambda),u_{b,\alpha}(\lambda))$. Since the zero at $E$ of the numerator of $G(z,p,p)$ is of order two, we obtain

 \begin{equation}\label{G}
\lim\limits_{z\rightarrow E}G_\alpha(z,p,p)=0
\end{equation}

If $u_{a,\alpha}'(E,b)=0$, we can assume that $u_{a,\alpha}(E,b)\not=0$. An analogous construction yields 
$$\frac{\partial}{\partial\lambda}\left[\frac{W_p(u_{a,\alpha}(\lambda),u_{b,\alpha}(\lambda))}{u_{a,\alpha}(\lambda,b)}\right]_{\lambda=E}=-\frac{\sin\theta}{u_{a,\alpha}(E,b)^2}\int_a^bu_{a,\alpha}(E,t)^2dt\not=0$$
\end{itemize}
and therefore (\ref{G}) follows.			
\qed		

The eigenvalues of an operator $L$ will be denoted by $\sigma_p(L)$ that is
$$\sigma_p(L):=\{r\in\mathbb{R}:L\varphi=r\varphi \mbox{ for some } \varphi\not\equiv 0,\,\, \varphi\in D(L)\}$$

\begin{theorem}\label{lemteoad}
	Let $E\in\mathbb{R}$ fixed. Then for the set
	$$A(E):=\{\alpha\in\mathbb{R}: E\in\sigma_p(H_{\alpha,p} )  \}$$
	there are two possibilities:
	\begin{itemize}
		\item[a)] $A(E)$ has at most one element.
		\item[b)] $A(E)=\mathbb{R}.$
	\end{itemize}
\end{theorem}		
\textit{Proof.}\\

If $u(p)=0$, then $u'(p+)-u'(p-)=\beta u(p)$, that is, $u\in D(H_{\beta,p}),\,\forall\beta\in\mathbb{R}$ and $H_{\beta,p} u=Eu$. Then, $E$ is eigenvalue of $H_{\beta,p}$.\\

If $u(p)\not=0$, from equation  (\ref{case1}) we get  $G_\alpha(E,p,p)=\infty$. By Corollary \ref{corocompa}, $\forall\beta\in\mathbb{R}\backslash\{\alpha\}$,
\begin{equation}\label{beta}
G_\beta(E;p,p)=\frac{1}{\beta-\alpha}
\end{equation}
 Since the right hand side of (\ref{beta}) is neither  zero  nor $\infty$, from Theorem \ref{polo} follows that  $E$ is not eigenvalue of $H_{\beta,p}$.\\
\hfill\qed	

\begin{remark}\label{observa1}
	Case $a)$ happens if the eigenvector $u$ associated to $E$ is such that $u(p)\not=0$, otherwise case $b)$ holds.
\end{remark}


\section{Sturm-Liouville Operators With Countably Many $\delta$-Point Interactions} \label{ndelta}
Up to this point, every result was about the operator with a single point interaction in the regular case. Now we will consider countable many point interactions. The results of the previous section will be used here.\\

Let  $-\infty\leq a<b\leq\infty$, $V\in L_{loc}^1(a,b)$ be a real valued function. Fix a discrete set $M$ of points accumulating at most at $a$ or $b$, $M:=\{x_n\}_{n\in I}\subset (a,b)$ where $I\subseteq\mathbb{Z}$ and let $\{\alpha_n\}\subset\mathbb{R}$.
We set $\alpha=\alpha_{n_0}$ and consider the formal differential expressions
\begin{equation}\label{tau}
\tau:=-\frac{d^2}{dx^2}+V \end{equation}
$$\tau_{\alpha,M}:=-\frac{d^2}{dx^2}+V(x)+\sum_{n\in I\backslash\{n_0\}}\alpha_n\delta(x-x_n)+\alpha\delta(x-x_{n_0})$$
The maximal operator $T_{\alpha,M}$ corresponding to $\tau_{\alpha,M}$ is defined by
$$T_{\alpha,M} f=\tau f$$
$$D(T_{\alpha,M})=\{f\in L^2(a,b):\,f,\,f'\mbox{ abs. cont in }(a,b)\backslash M,-f''+Vf\in L^2(a,b),$$
$$f(x_n+)=f(x_n-),\,f'(x_n+)-f'(x_n-)=\alpha_n f(x_n),\,\forall n\in I\}$$
 Analogous to what was done in the previous section, we introduce the following definitions.
\begin{definition}
Given $g\in L_{loc}^1(a,b)$ and $z\in\mathbb{C}$, we call $f$ a solution of $(\tau_{\alpha,M}-z)f=g$ if $f$ and $f'$ are absolutely continuous in $(a,b)\backslash M$ with $-f''+Vf-zf=g$ and $f(x_n+)=f(x_n-)$, $f'(x_n+)-f'(x_n-)=\alpha_n f(x_n),\, \forall n\in I$.
\end{definition}

\begin{definition}
We define the Wronskian of two solutions $u_1$ and $u_2$ of $(\tau_{\alpha,M} -z)f=0$ as in definition \ref{wron}, namely 
$$W_x(u_1,u_2)=u_1(x+)u'_2(x+)-u'_1(x+)u_2(x+)$$
\end{definition}

\begin{definition}
For $f,g\in D(T_{\alpha,M})$ we define the Lagrange bracket by
$$[f,g]_x=\overline {f(x+)}g'(x+)-\overline{f'(x+)}g(x+).$$
\end{definition}

The limits $[f,g]_a=\lim_{x\rightarrow a+}[f,g]_x$ and $[f,g]_b=\lim_{x\rightarrow b-}[f,g]_x$ exist. See Theorem 2.2 \cite{BSW}\\

A solution of $(\tau_{\alpha,M}-z)f=0$ is said to lie right (left) in $L^2(a,b)$, if $f$ is square integrable in a neighborhood of $b$ $(a)$.

\begin{definition}
\hfill
\begin{itemize}
\item[$i)$]
$\tau_{\alpha,M}$ is in the \textbf{limit circle case} (lcc) at $b$ if for every $z\in\mathbb{C}$ all solutions of $(\tau_{\alpha,M}-z)f=0$ lie right in $L^2(a,b)$.
\item[$ii)$] $\tau_{\alpha,M}$ is in the \textbf{limit point case} (lpc) at $b$ if for every $z\in\mathbb{C}$ there is at least one solution of $(\tau_{\alpha,M}-z)f=0$ not lying right in $L^2(a,b)$. 
\end{itemize}
The same definition applies to the endpoint $a$.
\end{definition}
According to the \textit{Weyl Alternative}, see \cite{BSW} Theorem 4.4, we have always either $i)$ or $ii)$. 


Consider the selfadjoint restriction $H_{\alpha,M}$ of $T_{\alpha,M}$ en $L_2(a,b)$ defined as
\begin{equation}\label{restriction}
H_{\alpha,M} f=\tau f
\end{equation}

\[
\begin{array}{ccc}
D(H_{\alpha,M})&=&\left\{f\in D(T_{\alpha,M}):
\begin{array}{c}
[v,f]_a=0 \mbox{     if $\tau_{\alpha,M}$ lcc at $a$ }\\ 
{[w,f]_b=0 }\mbox{     if $\tau_{\alpha,M}$ lcc at $b$}
\end{array}
\right\}.
\end{array}\]

Where $v$ and $w$ are non-trivial real solutions of $(\tau_{\alpha,M}-\lambda)v=0$ near $a$ and near $b$ respectively, $\lambda\in\mathbb{R}$. See Theorem 5.2 \cite{BSW} \\

\begin{definition}\label{regu}
	We say $\tau_{\alpha,M}$ is regular at $a$ if  $a$ is finite, $V\in L^1_{loc}[a,b)$ and $a$ is not an accumulation point of $M$. The same definition applies to the endpoint $b$.
\end{definition}

If $\tau_{\alpha,M}$ is regular at $a$, then  $\tau_{\alpha,M}$ is lcc at $a$ and the condition $[v,f]_a=0$ can be replaced by $$f(a)cos\psi+f'(a)\sin\psi=0$$ for $\psi\in[0,\pi)$. The same holds for $b$.\\

For $\gamma,\theta\in[0,\pi)$ and $[c,d]\subset[a,b]$, such that $[c,d]\cap M=\{x_{n_0}\}$, define the operator  $$H_{\alpha}^{\theta\gamma}:= H_{\alpha,x_{n_0}}$$ where $H_{\alpha,x_{n_0}}$ is as in formula (\ref{halfa}) of the previous section with $p=x_{n_0}$ and $J=[c,d]$. 


Let $E\in\mathbb{R}$ fixed and define
$$A(E):=\{\alpha\in\mathbb{R}:E\in\sigma_p(H_{\alpha,M}\}$$

\begin{lemma}\label{lema1del}
	There exist $\theta_0,\gamma_0\in[0,\pi)$ such that if $\alpha\in A(E)$, then $E\in\sigma_p(H_\alpha^{\theta_0\gamma_0})$.
\end{lemma}
\textit{Proof.}\\

If $\lambda_0\in A(E)$, then for some $\varphi\in D(H_{\lambda_0,M})$, $H_{\lambda_0,M}\varphi=E\varphi$.\\

Let us fix the points $\theta_0,\gamma_0\in[0,\pi)$ where
\begin{equation}\label{condiciones}
\begin{array}{c}
\varphi(c)cos\theta_0+\varphi'(c)\sin\theta_0=0
\\
\varphi(d)cos\gamma_0+\varphi'(d)\sin\gamma_0=0
\end{array}
\end{equation}

If $\alpha\in A(E)$ is such that  $\alpha=\lambda_0$, the assertion follows.\\

If $\alpha\in A(E)$ but $\lambda_0\not=\alpha$, then $H_{\alpha,M} \psi=E\psi$, for some $\psi\in D(H_{\alpha,M})$. Therefore, there exist $\theta,\gamma\in[0,\pi)$ which satisfy the boundary conditions  at $c$ and $d$ for $\psi$, similar to (\ref{condiciones}). If we prove that $\theta=\theta_0$ and $\gamma=\gamma_0$, then $H_\alpha^{\theta_0\gamma_0}\psi=E\psi$ and therefore $E\in\sigma(H_\alpha^{\theta_0\gamma_0})$.\\

Let us prove that $\gamma=\gamma_0$. The proof for $\theta$ is analogous.
\begin{itemize}
	\item[a)]Assume $\tau_{\alpha,M}$ is in the limit circle case at $b$.

The Wronkian satisfies $W_x(w,\varphi)=[w,\varphi]_x$ and $W_x(w,\psi)=[w,\psi]_x $ because $w$ is real. It is constant for $ x\in [d,b)$  since 
	$w,\psi$ and $ \varphi$ are solutions of $\tau_{\alpha,M} f=Ef$ in the interval $[d,b)$ 	because $x_{n_0}$ does not intersect $[d,b)$. By hypothesis, the functions $\varphi$ and $\psi$ satisfy the lcc condition at b. This implies
	$$0=[w,\psi]_b=\lim_{x\rightarrow b-}W_x(w,\psi)\qquad\mbox{and}\qquad 0=[w,\varphi]_b=\lim_{x\rightarrow b-}W_x(w,\varphi)$$
	Therefore $W_x(w,\psi)=W_x(w,\varphi)=0$  and then $W_x(\varphi,\psi)=0$. Thus $\varphi$ and $\psi$ are linearly dependent and $\varphi= K\psi$ for some constant $K\in\mathbb{R}$. Hence $\gamma=\gamma_0$. See Lemma 4.2 \cite{BSW}.
	\item[b)] Assume $\tau_{\alpha,M}$  is in the limit point case at $b$. If $\gamma_0\not=\gamma$ then $\varphi$ and $\psi$ are linearly independent in $[d,b)$, since if there exists a constant $K\in\mathbb{R}$ such that $\psi=K\varphi$ then $\gamma=\gamma_0$. Therefore every solution $f$ of $\tau_{\alpha,M} f=E f$ in $[d,b)$ can be written as $f=c_1\varphi+c_2\psi$. But, since $\varphi,\psi\in L^2(a,b)$, then $u\in L^2(a,b)$ and we get a contradiction to the limit point case.
\end{itemize}
\qed\\
An analogous argument was given in \cite{RR}. 

The following Theorem is a generalization of Corollary \ref{lemteoad}.\\

\begin{theorem}\label{teoad}
		Let $E\in\mathbb{R}$ fixed. Then for the set
	$$A(E):=\{\alpha\in\mathbb{R}:E\in\sigma_p(H_{\alpha,M})\}$$
		there are two possibilities:
		\begin{itemize}
			\item[a)] $A(E)$ has at most one element.
			\item[b)] $A(E)=\mathbb{R}.$
		\end{itemize}
	\end{theorem}		

	\textit{Proof.} By Theorem  \ref{lemteoad} for $E$ fixed one of the following holds
	\begin{itemize}
		\item There exist at most one $\alpha$ such that $E\in\sigma_p(H_\alpha^{\theta_0,\gamma_0})$ or
		\item $E\in\sigma_p(H_\alpha^{\theta_0,\gamma_0})$ for every $\alpha\in\mathbb{R}$ 
	\end{itemize}
	The assertion follows by Lemma  \ref{lema1del}.
	\\\hfill\qed

\begin{remark}\label{observa2}
	Case $a)$ happens if the eigenvector $u$ associated to $E$ is such that $u(x_{n_0})\not=0$, otherwise case $b)$ holds.
\end{remark}


\section{Random Sturm-Liouville Operators with $\delta$-Point Interactions}\label{mdelta}
In this section we use the previously obtained results to study the random case.
First the probability space $\Omega$ where the sequences of coupling constants live is constructed and then our random operators are defined.\\

The space of real valued sequences $\{\omega_n\}_{n\in I}$, where $I\subseteq \mathbb{Z}$, will be denoted by $\mathbb{R}^I$. We introduce a measure in $\mathbb{R}^I$ in the following way. Let $\{p_n\}_{n\in I}$ be a sequence of continuous probability measures in $\mathbb{R}$ ($p_n(\{r\})=0$ for any $r\in\mathbb{R}$) and consider the product measure  $\mathbb{P}=\times_{n\in I}p_n$ defined on the product $\sigma-$\'algebra $\mathcal{F}$ of $\mathbb{R}^I$ generated by the cylinder sets, that is, by the sets of the form $\{\omega:\omega(i_1)\in A_1,\dots,\omega(i_n)\in A_n\}$ for $i_1,\dots,i_n\in I$, where $A_1,\dots, A_n$ are Borel sets in $\mathbb{R}$. In this way a measure space $\tilde\Omega=(\mathbb{R}^I,\mathcal{F},\mathbb{P})$ is constructed. We consider then the completion of this space (subsets of sets of measure zero are measurable) $\tilde\Omega$ which will be denoted by $\Omega$. See chapter 1, section 1 in \cite{PF}.
\\

Let  $-\infty\leq a<b\leq\infty$, $V\in L_{loc}^1(a,b)$ be a real valued function. Fix a discrete set $M:=\{x_n\}_{n\in I}\subset (a,b)$ where $I\subseteq\mathbb{Z}$ and let $\omega=\{\omega(n)\}_{n\in I}\in\Omega$.
Consider the formal differential expression
$$\tau_{\omega}:=-\frac{d^2}{dx^2}+V(x)+\sum_{n\in I}\omega(n)\delta(x-x_n)$$
The maximal operator $T_{\omega}$ corresponding to $\tau_{\omega}$ is defined as before by
$$T_{\omega} f=\tau f$$
$$D(T_{\omega})=\{f\in L^2(a,b):\,f,\,f'\mbox{ abs. cont in }(a,b)\backslash M,-f''+Vf\in L^2(a,b),$$
$$f(x_n+)=f(x_n-),\,f'(x_n+)-f'(x_n-)=\omega(n)f(x_n),\,\forall n\in I\}$$
Assume the limit point occurs at $a$ or that $\tau_{\omega}$ is regular at $a$ (See Definition \ref{regu}) and the same possibilities for $b$.\\

For $\theta,\gamma\in [0,\pi)$ fixed, let $H_\omega^{\theta,\gamma}$ be the selfadjoint restriction of $T_\omega$ defined as 
\begin{equation}\label{halfa2}
H_{\omega}^{\theta,\gamma} f=\tau f
\end{equation}
\[
\begin{array}{ccc}
D(H_{\omega}^{\theta,\gamma})&=&\left\{f\in D(T_{\omega}):
\begin{array}{cc}
f(a)cos\theta+f'(a)sen\theta=0 &\mbox{in case $\tau_\omega$ regular at $a$}\\ 
{f(b)cos\gamma+f'(b)sen\gamma=0}&\mbox{in case $\tau_\omega$ regular at $a$}
\end{array}
\right\}
\end{array}\]
Notice that the index $\theta$ or $\gamma$ are meaningless if $\tau_\omega$ is lpc at $a$ or $b$. \\
In what follows instead of $H_\omega^{\theta,\gamma}$ we shall write $H_\omega$.

\begin{remark}
	
	One example where $\tau_{\omega}$ is lpc at both ends for all $\omega\in\Omega$ was given in Theorem 1 \cite{CCS}. There it was required that $I=\mathbb{Z}$, $V$ bounded and $\inf\limits_{n\in\mathbb{Z}} |x_{n+1}-x_n|>0$. \\
	
\end{remark}

\begin{definition}\label{ae}
For any $E\in\mathbb{R}$, we define
\begin{equation}\label{adee}
A(E):=\{\omega\in \Omega:E\in\sigma_p(H_\omega)\}
\end{equation}
For any measurable $B\subseteq A(E)$ and any $n\in I$, define
\begin{equation}\label{qne}
Q_{n,E}:=\{\omega\in B |\exists \,u_\omega\in D(H_\omega) ,\, H_\omega u_\omega = Eu_\omega\, \mbox{ and}\, u_\omega(x_n)\not= 0\} \end{equation}
\end{definition}

\begin{lemma}\label{qmes}
$Q_{n,E}$ is measurable and $\mathbb{P}(Q_{n,E})=0$.
\end{lemma}
\textit{Proof.}
Let
\[\begin{array}{ccc}

\chi_B(\omega)&=&\left\{\begin{array}{cc}
1&\mbox{ if }\omega\in B\\
0&\mbox{ if }\omega\not\in B

\end{array}\right.
\end{array}\]

If $\omega\in Q_{n,E}$, then from the definition of $Q_{n,E}$ follows $\chi_B(\omega)=1$. \\

Let $f:\mathbb{R}^{I\backslash \{n\}}\rightarrow [0,\infty)$. $$f(\tilde\omega):=\int_\mathbb{R} \chi_B(\omega)dp_{n}(\omega(n))$$

where $\tilde\omega=\sum\limits_{k\in I\backslash \{n\}}\omega(k)e(k)$. Here $e(k)=(e_m)_{m\in I}$ are the canonical vectors with entries $e_m =0$ if $k\not=m$ and $e_k=1$. The measurability of $f$ follows from Fubini's Theorem. (See Theorem 7.8 \cite{WR})

If $\omega=\sum\limits_{k\in I}\omega(k)e(k)\in Q_{n,E}$ then $f(\tilde\omega)=0$, where $\tilde\omega=\sum\limits_{k\in I\backslash \{n\}}\omega(k)e(k)$,  since $p_n$ is continuous and from theorem \ref{teoad}.\\

Hence $Q_{n,E}\subseteq [f^{-1}(\{0\})\times \mathbb{R}]\cap B$.\\

Now, using  Fubini,
$$\int_{f^{-1}(\{0\})\times\mathbb{R}}\chi_B(\omega)d\mathbb{P}=\int_{f^{-1}(\{0\})}d\mathbb{P}(\tilde\omega)\int_\mathbb{R}\chi_B(\omega)dp_n(\omega(n))=\int_{f^{-1}(\{0\})}f(\tilde\omega)d\mathbb{P}(\tilde\omega)=0$$

Then,
$$\int_{[f^{-1}(\{0\})\times\mathbb{R}]\cap B}\chi_B(\omega)d\mathbb{P}=0$$
and since $\chi_B(\omega)=1$ in $B$, then $\mathbb{P}([f^{-1}(\{0\})\times\mathbb{R}]\cap B)=0$.\\

Since the measure $d\mathbb{P}$ is complete, then any subset of a measurable set of measure zero is measurable with measure zero. Therefore $Q_{n,E}$ is measurable.
\\\hfill\qed

\begin{theorem}\label{isornot}
Let $E\in \mathbb{R}$ fixed and $B$ any measurable subset of $A(E)$. Then one of the following options hold:
\begin{itemize}
	\item[$i)$] $\mathbb{P}(B)=0$ 
	\item[$ii)$] $A(E)=\Omega$
\end{itemize}
\end{theorem}
\textit{Proof.} It will be enough to proof that if $ii)$ doesn't hold then $i)$ must hold.\\

Assume then that there exist $\omega_0\in \Omega$ such that $E$ is not eigenvalue of $H_{\omega_0}$. If $E$ is not eigenvalue of $H_\omega$, $\forall\omega\in\Omega$, then $\mathbb P(B)=0$ and the result follows. \\

Suppose now $\omega\in B$, then $E\in\sigma_p(H_\omega)$, i.e. there exist $u_\omega\in D(H_\omega)$ such that  $H_\omega u_\omega=Eu_\omega$. Then $\omega\in Q_{n,E}$, for some $n\in I$. This follows because if  $u_\omega(x_n)=0$ $\forall n\in I$,  then from the definition of $H_\omega$, $E$ must be an eigenvalue of $H_{\omega_0}$. Therefore
$$B\subset \bigcup_{n\in I}Q_{n,E}$$

Using lemma \ref{qmes}, then $\mathbb{P}(\bigcup\limits_{n\in I}Q_n)=0$, therefore the result follows.\\
\hfill\qed\\

For the next Corollary we denote by $H$ the operator $H_\omega$ defined in ( \ref{halfa2} ) with $\omega(n)=0$, $\forall n\in I$. This is just the selfadjoint operator generated by the differential expression $\tau$ in classical Sturm-Liouville theory without point interactions.

\begin{corolary}[cf. Remarks \ref{observa1}, \ref{observa2}] \label{alterna}
	\begin{itemize}
	\hfill
\item[a)]	If $E\not\in \sigma_p(H)$ then $\mathbb{P}(B)=0$  for any measurable subset $B$ of $\omega\in \Omega$ for which $E\in\sigma_p (H_\omega)$.
	
\item[ b)]	If $E\in \sigma_p(H)$ with $Hu=Eu$, then $A(E)=\Omega$ if and only if $u(x_n)=0$, $\forall n\in I$.
\end{itemize}
\end{corolary}	
\begin{proof}
	\hfill
	\begin{itemize}
		\item[a)] 
	If $E\not\in\sigma_p(H)$, then $\omega=(\dots,0,0,0,\dots)\not\in A(E)$. Therefore $A(E)\not=\Omega$ and the assertion follows from Theorem \ref{isornot}. 
	\item[b)] 		Suppose $E\in \sigma_p(H)$ with $Hu=Eu$.
	\begin{itemize}

		\item[$\Leftarrow)$] If  $u(x_n)=0$ $\forall n\in I$,  then from the definition of $H_\omega$, $E$ must be an eigenvalue of $H_{\omega}$ with eigenvector $u$, $\forall\omega\in\Omega$.
		\item[$\Rightarrow)$] From Lemma \ref{qmes}, $\mathbb{P}(\bigcup_{n\in I}Q_{n,E})=0$. Then $A(E)=\Omega\not\subseteq \bigcup_{n\in I}Q_{n,E}$. \\
		
		Take $\tilde{\omega}\in A(E)\backslash\bigcup\limits_{n\in I}Q_{n,E}$. There exists $u_{\tilde{\omega}}\in D(H_{\tilde{\omega}})$ such that
		$$(\tau_{\tilde{\omega}}-E)u_{\tilde{\omega}}=0 \qquad\mbox{and}\qquad u_{\tilde{\omega}}(x_n)=0,\,\forall n\in I$$
		Therefore  $u'_{\tilde{\omega}}(x_n+)-u'_{\tilde{\omega}}(x_n-)=\tilde\omega(n)u_{\tilde{\omega}}(x_n)=0$, $\forall n\in I$.	Hence $u'_{\tilde{\omega}}$ is continuous in $(a,b)$ and $(H-E)u_{\tilde{\omega}}=0$.\\
		Since all eigenvalues of $H$ are simple, see Theorem 8.29 (d)  \cite{JW}, then $u(x_n)=C u_{\tilde\omega}(x_n)=0$, $\forall n\in I$.
	\end{itemize}
\end{itemize}
\end{proof}

Then unless $E$ is an eigenvalue of $H$ and the point interactions are placed at the roots of eigenfuctions, we will have a ``small" set of operators $H_\omega$ will share the same eigenvalue $E$.\\

As another consequence of Theorem \ref{isornot} we get the following Corollary
\begin{corolary}
	Let $\{E_i\}_{i=1}^\infty$ be a sequence of real numbers and $B_i$ measurable subsets of $A(E_i)$. Assume there is no point $E\in\mathbb{R}$ which is eigenvalue of $H_\omega$ for all $\omega\in \Omega$, then 
	$$\mathbb{P}(\{\omega\in\bigcup_{i=1}^\infty B_i:\{E_i\}_{i=1}^\infty\cap\sigma_p(H_\omega)\not=\emptyset    \})=0$$
\end{corolary}
\textit{Proof.} By additivity of $\mathbb{P}$ and Theorem \ref{isornot}, we have
$$\mathbb{P}(\{\omega\in\Omega:\{E_i\}_{i=1}^\infty\cap\sigma_p(H_\omega)\not=\emptyset    \})=\mathbb{P}(\{\omega\in\Omega: \bigcup_{i=1}^\infty\left[\{E_i\}\cap\sigma_p(H_\omega)\right]\not=\emptyset)\}=$$
$$\mathbb{P}(\bigcup_{i=1}^\infty\{\omega\in\Omega: E_i\in\sigma_p(H_\omega))\}\leq\sum_{i=1}^\infty\mathbb{P}(\{\omega\in\Omega: E_i\in\sigma_p(H_\omega)\})=0$$
\qed

\subsection{Oscillation of Solutions}\label{sub1}
We shall use results about the oscillation of solutions of second order differential expressions. The location of zeros of eigenfunctions together with knowledge about the positions of the point interactions, will help us to understand when option $b)$ in Theorem \ref{isornot} happens.\\

In this subsection $\tau$ is as in equation ( \ref{tau} ) and $A(E)$ is defined as in Definition \ref{ae} (  \ref{adee} ) that is as the set of $\omega\in\Omega$ such that $H_\omega$ share the common eigenvalue $E$.

\begin{definition}[See Section XI.6 in \cite{HP}]
	The equation  $$(\tau-E)f=0 $$
	is said to be nonoscillatory on an interval $J$ if every solution has at most a finite number of zeros on $J$.\\
	
	If $t=b$ is a (possibly infinite) endpoint of $J$ which does not belong to $J$, then the equation is said to be nonoscillatory at $t=b$ if every solution has a finite number of zeros in $J$ or if the zeros do not accumulate at $b$.
\end{definition}

\begin{lemma}\label{lem2}
	If $A(E)=\Omega$, then there exists a solution $u$ of $(\tau-E)f=0$ such that $u(x_n)=0$, $\forall n\in I$.
\end{lemma}
\begin{proof}
	If $A(E)=\Omega$, then there exists $u$ such that  $H u= Eu$, where  $H$ is the operator $H_\omega$ with $\omega(n)=0$, $\forall n\in I$. From Corollary \ref{alterna} (b) the assertion follows.
\end{proof}

\begin{theorem}\label{lyap}
	Let $V$ be the potential appearing in the expression ( \ref{tau} ). Assume  $|V(x)|\leq K$ for all $x\in (a,b)$. Let $J$ be an interval such that   $$|J|\leq\frac{2}{\sqrt {K+|E|}}$$ 
	where $|J|$ denotes the length of the interval.  Assume $J\cap M$ has at least two elements.
	Then $\mathbb{P}(B)=0$ for any measurable subset $B$ of $A(E)$.
\end{theorem}
\begin{proof}
	Suppose there exists a measurable subset $B$ of $A(E)$ such that $\mathbb{P}(B)>0$, then from Theorem \ref{isornot}, $A(E)=\Omega$.\\

 By Lemma \ref{lem2}, there exists a solution $u$ of $(\tau-E)f=0$ such that $u(x_n)=0$, $\forall n\in I$.
	 Using a theorem due to Lyapunov, see Theorem 3.9 of \cite{DHU} and Corollary 5.1 of \cite{HP}, the interval $J$ is disconjugate , i.e.  there is at most one zero of any solution of $(\tau-E)f=0$ in the interval $J$, since $u$ is solution this is a contradiction, hence $\mathbb{P}(B)=0$ for any measurable subset $B$ of $A(E)$.
\end{proof}

In the last Theorem observe that the larger  $|E|$ is, the smaller $|J|$ has to be. This corresponds to the fact that the solutions oscillate faster if the energy is high.

\begin{theorem}\label{non}
	Suppose $(\tau-E)f=0$ is nonoscillatory in $(a,b)$ and the set of interactions $M$ is a countable set. Then $\mathbb{P}(B)=0$ for any measurable subset $B$ of $A(E)$.
\end{theorem}
\begin{proof}
		Suppose there exists a measurable subset $B$ of $A(E)$ such that $\mathbb{P}(B)>0$, then from Theorem \ref{isornot}, $A(E)=\Omega$.\\
	
	By Lemma \ref{lem2}, there exists a solution $u$ of $(\tau-E)f=0$ such that $u(x_n)=0$, $\forall n\in I$.
	The equation $(\tau-E)f=0$ is nonoscillatory i.e. any solution in the interval $(a,b)$ has at most a finite number of zeros, . Since $u$ is solution this is a contradiction, hence $\mathbb{P}(B)=0$ for any measurable subset $B$ of $A(E)$.
\end{proof}

There are several conditions in the literature which allow us to conclude that our problem is nonoscillatory. Applying a Theorem of Hille, Theorem 3.1  \cite{DHU}, we get the following result 

\begin{theorem}
	If $V$ is continuous in $[a,\infty)$, $V(x)\leq E$, $\int_a^\infty (E-V(x) )dx <\infty $ and
	$$\limsup_{x\rightarrow\infty}x\int_x^\infty (E-V(t))dt<\frac{1}{4}$$
	then $(\tau-E)f=0$ is nonoscillatory at $[a,\infty)$.
\end{theorem}

Finer estimates on the number of zeros can be used too, as the following result shows.

\begin{theorem}
	Let $V$ continuous in $[0, T]$ and $|V(x)|\leq  K$. If the number of points in $M\cap [0,T]$ is greater or equal to
	$$\frac{T\sqrt{|E|+K}}{2}+1$$
	then, $\mathbb{P}(B)=0$ for any measurable subset $B$ of $A(E)$.
\end{theorem}
\begin{proof}
Using  Corollary 5.2 in \cite{HP} we see that the number of zeros of any solution of $(\tau-E)f=0$ is less than
$$\frac{T\sqrt{|E|+K}}{2}+1$$
and then the proof follows as in Theorem \ref{lyap}.
\end{proof}

\subsection{Measurable Operators}\label{sub2}
 Now we introduce condition of measurability for the family of operators $H_\omega$.
\begin{definition}[See Lemma 1.2.2 in \cite{PS}, Proposition 3 in \cite{KM}]\label{meafam}
A family $\{S_\omega\}_{\omega\in\Omega}$ of  selfadjoint operators in a Hilbert space $\mathfrak{H}$ is called measurable  if the mappings
$$\omega\rightarrow<\varphi,E_\omega(\lambda)\psi>$$
are measurable for all $\varphi,\,\psi\in \mathfrak{H}$, where $E_\omega(\lambda)$ is the corresponding resolution of identity of $S_\omega$. 
\end{definition}

\begin{theorem}\label{meas}[Communicated to us by Peter Stollmann]
Let $$A(E):=\{\omega\in\Omega:E\in\sigma_p(S_\omega)\}$$ 
 as in Definition \ref{ae}. If $\{S_\omega\}_{\omega\in\Omega}$ is a measurable family of operators defined in a separable Hilbert space $\mathfrak H$, then $A(E)$ is measurable.
\end{theorem}
\textit{Proof.} Let $\{\psi_n\}_{n\in\mathbb N}$ be a countable dense subset of $\mathfrak H$.

Observe that 
\begin{equation}\label{an}
A (E)=   \bigcup_{n\in\mathbb N}A_n
\end{equation}
where $$A_n:= \{\omega\in \Omega : E_\omega(\{E\})\psi_n\not=0\}$$

The set on the right hand side of (\ref{an}) is contained in $A(E)$  since
$A(E)=\{\omega\in\Omega | E_\omega(\{E\}) \not=0\}$.
 To prove the other inclusion, let $\omega\in A(E)$ and assume that for all $n$, $E_\omega(\{E\})\psi_n=0.$ For any $x\in\mathfrak{H}$ we have
	$$<E_\omega(\{E\})x,\psi_n>=<x,E_\omega(\{E\})\psi_n>=0.$$
     Since $\{\psi_n\}$ is dense, $E_\omega(\{E\})x=0$ and $E_\omega(\{E\})=0$, which is a contradiction to $\omega\in A(E)$. Therefore there is $n_0$ such that $E_\omega(\{E\})\psi_{n_0}\not=0$ and $\omega\in\bigcup\limits_{n\in\mathbb N}A_n$.
     
     We shall now prove that the sets $A_n$ are measurable.
     
Since $S_\omega$ is measurable, the function $f_n$ defined as  $\omega\rightarrow f_n(\omega):=<\psi_n, E_{\omega}(\{E\})\psi_n>$ is measurable for each $n$. We get that $\omega\in A_n^c$ if and only if
$$f_n(\omega)=<\psi_n, E_{\omega}(\{E\})\psi_n>=\|E_\omega(\{E\})\psi_n\|^2=0.$$
Thus 

$$A_n^c=\{\omega | E_\omega(\{E\})\psi_n=0\}=f_n^{-1}(\{0\}).$$
It follows that  $A_n^c$ and therefore  $A_n$ are measurable sets. Hence $A(E)$ is a countable union of measurable sets, thus measurable.
\\\hfill\qed

Using Theorem \ref{meas} we obtain the following Corollary to Theorem \ref{isornot}.
Let $\{H_\omega\}_{\omega\in\Omega}$ the family of operators introduced in  ( \ref{halfa2} ).
\begin{corolary}
  Assume that the family $\{H_\omega\}_{\omega\in\Omega}$ is measurable. For fixed $E\in \mathbb{R}$, one of the following options hold:
\begin{itemize}
	\item[$i)$] $\mathbb{P}(A(E))=0$ 
	\item[$ii)$] $A(E)=\Omega$
\end{itemize}
\end{corolary}
\textit{Proof.} Take $B=A(E)$ in Theorem \ref{isornot}.
\hfill\qed\\

As example of a measurable family, let us mention the operators generated by the formal differential expression 
$$\tau_{\omega}:=-\frac{d^2}{dx^2}+\sum_{n\in I}\omega(n)\delta(x-x_n)$$
where $\omega(n)$ is a stationary metrically transitive random field satisfying $|\omega(n)|\leq C<\infty$, see \cite{KM}. In particular we can take $\omega(n)$ to be independent identically distributed random variables. \\

Since the operator generated by $-\frac{d^2}{dx^2}$ without point interactions does not have eigenvalues, we can apply Corollary \ref{alterna} and obtain $\mathbb{P}(A(E))=0$. We get in this case a proof of a result due to Pastur which says that the probability of any fixed $\lambda\in\mathbb{R}$ being an eigenvalue of finite multiplicity of a metrically transitive operator is zero, see Theorem 3 in \cite{LP} and Theorem 2.12 in \cite{PF}.
		
\section{Sturm-Liouville Operators with $\delta'$-Point Interactions}\label{delta2}		
Now we consider operators with $\delta'$-interactions and show how analogous results can be obtained.
Let  $-\infty\leq a<b\leq\infty$, $V\in L_{loc}^1(a,b)$ be a real valued function. Fix a discrete set $M:=\{x_n\}_{n\in I}\subset (a,b)$ where $I\subseteq\mathbb{Z}$ and let $\omega=\{\omega(n)\}_{n\in I}\in\Omega$, where $\Omega$ is defined as in Section 3.
Consider the formal differential expression
$$\tau_{\omega}:=-\frac{d^2}{dx^2}+V(x)+\sum_{n\in I}\omega(n)\delta'(x-x_n)$$
The maximal operator $T_{\omega}$ corresponding to $\tau_{\omega}$ is defined by
$$T_{\omega} f=\tau f=-\frac{d^2f}{dx^2}+Vf$$
$$D(T_{\omega})=\{f\in L^2(a,b):\,f,\,f'\mbox{ abs. cont in }(a,b)\backslash M,-f''+Vf\in L^2(a,b),$$
$$f'(x_n+)=f'(x_n-),\,f(x_n+)-f(x_n-)=\omega(n)f'(x_n),\,\forall n\in I\}$$
 The construction is similar to what we did in Section \ref{mdelta}, but notice the change of the conditions at the points $x_n$.\\
 \begin{definition}\label{sol2}
 A function $f$ is a solution of $(\tau_\omega-\lambda)f=0$ if $f$ and $f'$ are absolutely continuous in $(a,b)\backslash M$ with $-f''+Vf-\lambda f=0$ and $f'(x_n+)=f'(x_n-)$, $f(x_n+)-f(x_n-)=\omega(n) f'(x_n),\, \forall n\in I$.
 \end{definition}
 
Assume the limit point occurs at $a$ or that $\tau_{\omega}$ is regular at $a$ (See Definition \ref{regu}) and the same possibilities for $b$.\\

For $\theta,\gamma\in [0,\pi)$ fixed, let $H_\omega^{\theta,\gamma}$ be the selfadjoint restriction of $T_\omega$ defined as 
\begin{equation}\label{halfa3}
H_{\omega}^{\theta,\gamma} f=\tau f
\end{equation}
\[
\begin{array}{ccc}
D(H_{\omega}^{\theta,\gamma})&=&\left\{f\in D(T_{\omega}):
\begin{array}{cc}
f(a)cos\theta+f'(a)sen\theta=0 &\mbox{in case $\tau_\omega$ regular at $a$}\\ 
{f(b)cos\gamma+f'(b)sen\gamma=0}&\mbox{in case $\tau_\omega$ regular at $a$}
\end{array}
\right\}
\end{array}\]
Notice that the index $\theta$ or $\gamma$ are meaningless if $\tau_\omega$ is lpc at $a$ or $b$. \\
In what follows instead of $H_\omega^{\theta,\gamma}$ we shall write $H_\omega$.
\\
Similarly to what has been done before one can prove for this $H_\omega$ with $\delta'$ interactions the following Theorem.
\begin{theorem}\label{isornot2}
Let $E\in \mathbb{R}$ fixed and $B$ any measurable subset of 
$$A(E):=\{\omega\in\Omega:E\in\sigma_p(H_\omega )\}.$$
 Then one of the following options hold:
\begin{itemize}
	\item[$i)$] $\mathbb{P}(B)=0$ 
	\item[$ii)$] $A(E)=\Omega$
\end{itemize}
\end{theorem}
\begin{proof}
The proof follows closely the arguments given in Sections \ref{1delta}, \ref{ndelta} and \ref{mdelta}. The Wronskian can be defined for solutions of $(\tau_\omega-\lambda)u=0$ with $\delta'$-interactions and the continuity at the points $x_n$  holds as in  Lemma \ref{wcont}. The main modification in Section \ref{1delta} is the 
use of  $\tilde G$ introduced below, instead of $G$ given in Definition \ref{Ge}.
For $z\in\mathbb{C}$ let
$$\tilde G_\alpha(z,x,x):=\frac{u'_{a,\alpha}(z,x)u'_{b,\alpha}(z,x)}{W_x(u_{a,\alpha}(z),u_{b,\alpha}(z))}$$
where $u_{a,\alpha}$ and $u_{b,\alpha}$ are as in  Definition \ref{solution} but satisfying the conditions of the $\delta'-$interaction $f'(p+)=f'(p-)$ and $f(p+)-f(p-)=\alpha f'(p)$.\\

If $x\leq p$, $u_{a,0}(x)=u_{a,\alpha}(x)$ and if $x\geq p$, $u_{b,0}(x)=u_{b,\alpha}(x)$. \\\\
Now, from the condition at $p$
$$u_{a,\alpha}(p+)=u_{a,\alpha}(p-)+\alpha  u'_{a,\alpha}(p)=u_{a,0}(p-)+\alpha  u'_{a,0}(p)=u_{a,0}(p+)+\alpha u'_{a,0}(p).$$

Using this in $\tilde G_\alpha$ we get
$$\tilde G_\alpha(z,p,p)=\frac{u'_{a,\alpha}(p) u'_{b,\alpha}(p)}{W( u_{a,\alpha}, u_{b,\alpha})}=\frac{ u'_{a,\alpha}(p) u'_{b,\alpha}(p)}{u_{a,\alpha}(p+) u'_{b,\alpha}(p)-u'_{a,\alpha}(p) u_{b,\alpha}(p+)}=$$
$$=\frac{u'_{a,0}(p)u'_{b,0}(p)}{u_{a,0}(p+)u'_{b,0}(p)+\alpha u'_{a,0}(p)u'_{b,0}(p)-u'_{a,0}(p)u_{b,0}(p+)}
$$
Then
$$\tilde G_\alpha(z,p,p)=\frac{u'_{a,0}(p)u'_{b,0}(p)}{W(u_{a,0},u_{b,0})\left(1+\alpha\frac{u'_{a,0}(p)u'_{b,0}(p)}{W(u_{a,0},u_{b,0})}\right)}=\tilde G_0(z,p,p)\frac{1}{1+\alpha \tilde G_0(z,p,p)}$$

and we obtain the characterization of the eigenvalues given in Theorem \ref{polo} and then Theorem \ref{lemteoad} for $\delta'-$interactions.\\
 The results in Section \ref{ndelta} hold for $\delta'-$interactions practically without modifications. In Section \ref{mdelta} we modify  Definition \ref{ae} ( \ref{qne} )   by setting $u'_\omega(x_n)\not=0$ instead of $u_\omega(x_n)\not=0$ and obtain Lemma \ref{qmes} and Theorem \ref{isornot} for $\delta'-$interactions. Hence the result follows. \\
\end{proof}

\begin{remark}
 Mixed situations where  $\delta$ and $\delta'$ interactions  are present, can be treated  with the arguments given above.
\end{remark}
	
\section*{Acknowledgements}
We are indebted to F. Gesztesy, W. Kirsch, P. Stollmann, G. Stolz and G. Teschl for very useful and stimulating discussions. \\

This work was partially supported by project PAPIIT IN 110818.

	\bibliographystyle{plain}
	\bibliography{biblio}

\end{document}